\documentclass[12pt]{article}
\title{Amenability of quadratic automaton groups}
\author{Gideon Amir, Omer Angel, Balint Vir\'ag}
\date{October 2021}

\usepackage{amsmath,amssymb,amsfonts,amsthm}

\usepackage{hyperref}
\hypersetup{
    colorlinks=true,
    linkcolor=black,
    citecolor=black,
    urlcolor=blue,
    pdfborder={0 0 0}
}            

\usepackage{graphicx}
\usepackage{xcolor}
\usepackage{enumitem}

\usepackage[font=sf, labelfont={sf,bf}, margin=1cm]{caption}

\usepackage{cleveref}
  \crefname{theorem}{Theorem}{Theorems}
  \crefname{thm}{Theorem}{Theorems}
  \crefname{lemma}{Lemma}{Lemmas}
  \crefname{lem}{Lemma}{Lemmas}
  \crefname{remark}{Remark}{Remarks}
  \crefname{prop}{Proposition}{Propositions}
  \crefname{defn}{Definition}{Definitions}
  \crefname{corollary}{Corollary}{Corollaries}
  \crefname{section}{Section}{Sections}
  \crefname{figure}{Figure}{Figures}

\newtheorem{thm}{Theorem}[section]
\newtheorem{lemma}[thm]{Lemma}
\newtheorem{corollary}[thm]{Corollary}
\newtheorem{prop}[thm]{Proposition}

\newtheorem{conj}[thm]{Conjecture}
\theoremstyle{definition}
\newtheorem{remark}[thm]{Remark}

\newtheorem{question}[thm]{Question}

\numberwithin{equation}{section}

\newcommand{\Z}{\mathbb Z}

\newcommand{\R}{\mathbb R}
\newcommand{\N}{\mathbb N}

\newcommand{\Aut}{\operatorname{Aut}}

\newcommand{\cG}{\mathcal{G}}
\newcommand{\cE}{\mathcal{E}}
\newcommand{\cM}{\mathcal{M}}
\renewcommand{\hat}{\widehat}
\newcommand{\Res}{\operatorname{Res}}
\newcommand{\mbar}{{\overline m}}
\begin{document}

\maketitle

\begin{abstract}
  We give lower bounds for the electrical resistance between vertices in the Schreier graphs of the action of the linear (degree 1) and quadratic (degree 2) mother groups on the orbit of the zero ray.
  These bounds, combined with results of \cite{JNS} show that every quadratic activity automaton group is amenable.
  The resistance bounds use an apparently new ``weighted'' version of the Nash-Williams criterion which may be of independent interest.
\end{abstract}

\section{Introduction}

Automaton groups are a rich family of groups, with a simple definition which exhibit rich behaviour.
They include many groups with interesting properties, including
the Grigorchuk group of intermediate growth,
the Basilica group, the Hanoi towers group, lamplighter groups, and many others.
Automaton groups are certain subgroups of the automorphism group $\Aut(T_m)$ of the rooted infinite $m$-ary tree for some $m$.
For any vertex $v\in T_m$, and automorphism $g$, there is an induced action of $g$ on the sub-tree above $v$.
This action is called the \textbf{section} of $g$ at $v$, denoted $g_v$.
An action maps a sub-tree to another sub-tree, but since a sub-tree is isomorphic to the whole tree, a section can be viewed naturally as an automorphism of the whole tree.
Automaton groups are finitely generated sub-groups of $\Aut(T_m)$ with generators for which all sections are also among the generating set
(see \cite{Nek05} for detailed definitions).

The \textbf{activity} of $g\in\Aut(T_m)$ at level $n$, denoted $a_n(g)$, is the number of $v$ in level $n$ of the tree such that $g_v$ is not the identity.
Automaton groups have the property that for any $g$ the activity sequence $a_n(g)$ grows either polynomially or exponentially.
An automaton group $\Gamma$ is said to have \textbf{degree $d$}, if every $g\in\Gamma$ has activity $a_n(g) = O(n^d)$.
Activity of $g$ was introduced by Sidki \cite{Sidki00} as a measure of the complexity of the action of $g$, and the degree of an automaton group as a measure of the complexity of the group.

There exist exponential activity automaton groups that are isomorphic to the free group \cite{GlasnerMozes05,Vorobets07}.
However, one expects polynomial activity automaton groups to be smaller.
In particular, in contrast with most examples of finitely generated non-amenable groups, \cite{Sidki00} showed that polynomial activity automaton groups have no free subgroups.
This prompted \cite{Sidki04} to ask the following natural question:

\begin{question}
  Are all polynomial activity automaton groups amenable?
\end{question}

This was answered affirmatively for degree $0$ in \cite{BKN} and for degree $1$ in \cite{AAV}.
These results were also reproved in \cite{JNS} (see the discussion below).
Our main result resolves Sidki's question for degree $2$:

\begin{thm}\label{T:main}
  Every automaton group of degree $2$ is amenable.
\end{thm}

For degrees $0$ and $1$, the proofs of \cite{BKN,AAV} proceed as follows.
First, for each degree $d$ and $m$, a certain specific automaton group acting on $T_m$, called the \textbf{mother group} is constructed.
It is then shown that every automaton group $\Gamma$ of degree $d$ is isomorphic to a subgroup of the mother group of degree $d$ for some $m'$.
Next, it is proved that the mother groups of degree $0$ and $1$ are Liouville with respect to a carefully chosen random walk on them.
Since the Liouville property implies amenability, and amenability is inherited by subgroups, this implies amenability of all bounded or linear activity automaton groups.

It is shown in \cite{AV} that for $d\ge 3$ the mother groups are not Liouville.\footnote{Except for the case $d=3$ and $m=2$ which remains open.}
Thus the method of \cite{BKN} cannot be extended to degree $d>2$.
This raises the natural question of the Liouville property for degree $d=2$.

\begin{conj}[\cite{AAV}]
  The mother groups of degree $2$ are Liouville w.r.t. some (or even every) random walk on them.
  Moreover, the same holds for every automaton group of degree $2$.
\end{conj}

The Liouville property of the mother groups is established in the papers above by showing that a certain random walk on the group has sublinear entropy.
By results of Kaimanovich and Vershik \cite{KV} sublinear entropy growth is equivalent to having the Liouville property (w.r.t. this random walk).
In a sense, entropy bounds can be thought of quantitative versions of the Liouville property.
One should note that -- while amenability is inherited by subgroups -- it is not known whether the Liouville property passes onto subgroups.
Consequently, the results of \cite{BKN,AAV} do not imply that all automata groups of degrees $0$ or $1$ are Liouville, nor that the mother groups are Liouville w.r.t. other generating sets.

The fact that degree $0$ automata groups are Liouville w.r.t. any measure on them was proved in \cite{AAMBV} by giving explicit entropy bounds for random walks on these groups.
These entropy bounds come from resistance lower bounds in the Schreier graphs associated to the action of the automata group on a ray of the tree.
To get such lower bounds it is enough to attain lower bounds on the resistance for the Schreier graphs of the mother groups, since resistance can only increase when going into subgraphs.
Thus resistance estimates on the Schreier graph of the mother groups imply entropy estimates and the Liouville property for bounded automata.
We believe that a similar approach can be used to show that automata groups of degree $1$ also have the Liouville property for any measure.
For higher degree automata groups the situation is different.
In \cite{AV} it was shown that the Schreier graphs for degree $3$ and up mother groups are transient.
This was used to show (as noted above) that these groups do not have the Liouville property.

Upper and lower bounds for resistances in the Schreier graphs of the mother groups were given in \cite{AV} and \cite{AAMBV}.
These bounds are tight for degree $0$ mother groups, were enough to deduce transience for degree $3$ and up mother groups.
For degrees $1$ and $2$ there are significant gaps between the upper and lower bounds on resistances.
In particular, these bounds were not enough to deduce recurrence of the Schreier graphs for degree $2$ mother groups.

\medskip
Since the Liouville property is harder to establish for $d=2$ and false for $d>2$, new methods are needed for further progress on Sidki's conjecture.
In \cite{JNS}, Juschenko, Nekrashevych, and de la Salle proved that (under some general conditions), if the action of a group $G$ on a set $X$ is significant enough, and the Schreier graph of the action is recurrent, then the group is amenable.
In the context of polynomial activity automaton groups,
they consider the action of the group on the orbit of the $0$-ray of the tree, and show that if the Schreier graph of this action is recurrent, then the group is amenable.
This yielded a second proof of the amenability of degree $0$ and $1$ automata groups that does not pass via the Liouville property.
\cref{T:main} is a corollary of the following result, together with \cite[Theorem 5.2]{JNS}.

\begin{thm}\label{T:d2rec}
  For any degree 2 automaton group, the Schreier graph of its action on a ray of $T_m$ is recurrent.
\end{thm}

\medskip

The bulk of the work in this paper is actually in a more general context of groups of automorphisms of a spherically symmetric tree (i.e. a tree where every vertex at distance $k$ from the root has the same number $m_k$ of children).
Such groups were used in \cite{Bri13,AV,Amir} to construct groups where the random walk has varied speeds and entropy growth.

\subsection{Mother groups on spherically symmetric trees}

Let $\mbar = (m_0,m_1,\dots)$ be some infinite \emph{bounded} sequence with $m_i\in\{2,3,\dots\}$.
We consider the spherically symmetric tree
$T_\mbar$ defined as follows.  At level $0$ there is a single vertex
$\emptyset$ (the root).  Each vertex at level $i$ has $m_i$ children at
level $i$, so that the size of level $\ell$ is $\prod_{i<\ell} m_i$.  A
vertex at level $\ell$ is naturally encoded by a word $x_\ell\dots x_1 x_0$
where $x_i\in[m_i] = \{0,1,\dots,m_i-1\}$.  Since we will later have a
group acting on $T_\mbar$ on the right, it is more useful to write the digits with $x_0$ on the right.
The set of \textbf{ends of the tree} $T_\mbar$, denoted $\cE$, are the infinite rays in $T_\mbar$, and are naturally encoded by infinite sequences (which we again write with $x_0$ on the right) $\dots x_2 x_1 x_0$, with $x_i\in[m_i]$.
The subset of ends with only finitely many nonzero digits is denoted $\cE_0$.

We remark that the case of $\mbar$ constant is already new and of interest.
This case is of particular significance since the corresponding groups (as defined below) are the automaton groups discussed above.
The confused reader may well restrict to the case where $\mbar$ is the constant sequence, and the tree is the $m$-ary tree, without losing much.

We consider \textbf{automorphisms} of the rooted tree (which preserve the root $\emptyset$, and hence each level).
(For some sequences $\mbar$ such as $(3,2,2,2,\dots)$ there are automorphisms which do not preserve the root, but we do not consider such automorphisms in this work.)
An automorphism acts naturally on the set of ends of the tree, and is
determined by this action.
A bijection $f$ of the set of ends with itself corresponds to an automorphism of the tree if for every $i$, the $i$th
digit of $f(x)$ is determined by $x_i,\dots,x_0$.

Towards defining our groups, we need notation for the locations of non-zero digits in an end of the tree.
For a (finite or infinite) word $x$, let $\ell_{-1}(x) = -1$, and inductively let
\[
  \ell_t(x) = \inf \{n>\ell_{t-1}(x) : x_n \neq 0\}.
\]
For some fixed degree $d$, the \textbf{mother group} of degree $d$, denoted $\cM_{d,\mbar}$, is a subgroup of the automorphism group with the following set of generators.
Each generator is specified by a degree $t\in\{-1,0,\dots,d\}$, and a sequence $(\sigma_i)_{i\leq\max{\mbar}}$, where $\sigma_i$ is a permutation in the symmetric group $S_{i}$ for each $i$.
For an end $x$, let $k = 1+\ell_t(x)$.  The generator applies $\sigma_{m_k}$ to the digit $x_k$, and leaves all other digits unchanged.
If $k=\infty$ (which happens for some ends in $\cE_0$), then $x$ is a fixed point of the generator.

As an example, suppose $x=\dots201300010020$.
Then $\ell_{0}=1$,  $\ell_1=4$,  $\ell_2=8$, etc.
If (for whatever value of $m_9$), $\sigma_{m_9}$ maps 1 to 0 then the corresponding generator with $t=2$ will map this $x$ to $\dots200300010020$.

Note that the subset $\cE_0$ of ends is preserved by all actions of generators of the mother group, and hence by actions of the group.
Moreover, $\cE_0$ is dense in the set of all ends, and so the action on $\cE_0$ determines an automorphism of the tree.
Finally, we remark that the mother groups act transitively on $\cE_0$.
(This is not hard but requires some observation and is also the basis of some mechanical puzzles such as the \emph{Chinese rings}.)
For $d=1,2$ these groups are referred to as the \textbf{linear} and \textbf{quadratic} mother groups respectively.

Recall that the \textbf{Schreier graph} for the action of a group $G$ generated by $S$ on a set $A$ is the graph with vertex set $A$ and an edge $(x,x\sigma)$ for any $x\in A$ and $\sigma\in S$.
Our main object of study in this work is the Schreier graph $\cG_{d,\mbar}$ for the action of $\cM_{d,\mbar}$ on $\cE_0$.
It is not hard to see that $\cG_{d,\mbar}$ is connected, and is the connected component of the $0$ ray in the Schreier graph for the action on $\cE$.
We shall also consider the Schreier graphs for the action on level $n$ of the tree, which will be denoted $\cG_{d,\mbar,n}$.

\subsection{Results for mother groups}

\begin{thm}\label{T:recurrent}
  For $d\leq 2$ and any bounded sequence $\mbar$, the Schreier graph
  $\cG_{d,\mbar}$ is recurrent.
\end{thm}

We expect other components of the Schreier graph to have a very similar geometry to the component on $\cE_0$, and in particular to also be recurrent.
This is not needed for the application to amenability of the mother groups, and the combinatorial ingredients in the analysis of the geometry of the graph are easier for $\cE_0$, and so we restrict our attention to that component.

In the case $d=0$, the Schreier graph has been analyzed in \cite{AV}.
When $d=0$ and $m_i \equiv 2$, the graph is simply the half-line $\N$.
(Other components of the Schreier graph in this case are isomorphic to $\Z$.)
For general $\mbar$ it is easily seen to be recurrent as it contains infinitely many cut-sets of bounded size.
Resistances in $\cG_{0,\mbar}$ are studied in \cite{AV}.

For $d=1,2$ recurrence of $\cG_{d,\mbar}$ is a direct consequence of the quantitative estimates in the following theorem, which require some additional notation.
In \cref{sec:graphs} we describe an explicit projection $\hat\pi:\cG_{d,\mbar} \to \N$ with the following properties: The only vertex with $\hat\pi(v)=0$ is the $0$ ray, and each $n\in\N$ has a finite non-empty pre-image.
In the case $m_i\equiv 2$, $\hat\pi$ is a bijection.

\begin{thm}
  \label{T:resistance}
  Fix a bounded sequence $\mbar$.
  There exists a constant $C$, depending only on $\sup\mbar$ such the following holds.
  For any $0<s<t$ the resistance in $\cG_{d,\mbar}$ satisfies
  \[
    \Res\big(\hat\pi^{-1}[0,2^s), \hat\pi^{-1}[2^t,\infty)\big) \geq
    \begin{cases}
      C(t-s) & \text{for } d=1, \\
      C(\log t-\log s) & \text{for } d=2. \\
    \end{cases}
  \]
\end{thm}

\begin{remark}
Note that by monotonicity, if $a<2^s<2^t<b$ then
\[  \Res\big(\hat\pi^{-1}[0,a), \hat\pi^{-1}[b,\infty)\big) \geq \Res\big(\hat\pi^{-1}[0,2^s), \hat\pi^{-1}[2^t,\infty)\big). \]
Thus \cref{T:resistance} implies a similar bound for such resistances
(i.e. $\log(b/a)$ and $\log\log(b/a)$ in the two cases respectively) as long as $b\ge 4a$.
For $b$ close to $a$, the result might fail.
Indeed, if $b=a+1$ then the resistance can be of order $a^{-\delta}$ for some $\delta$ depending on $\bar{m}$, which can be large if $\bar{m}$ has large entries.
\end{remark}


As mentioned in the introduction, any automaton group of degree $d$ is conjugate to a sub-group of the mother group of the same degree $d$, possibly on a larger alphabet \cite{AAV}.
A similar statement holds also for general spherically symmetric trees \cite{Bri09}.
Since resistances in subgraphs are larger than resistances in a graph, we get the following corollary, which in turn implies amenability of the groups.

\begin{corollary}
  The Schreier graph for the natural action of any automaton group of degree at most 2 on the ends of the regular tree has a recurrent component.
\end{corollary}

\paragraph{Structure of the paper.}
In \cref{sec:graphs} we give a combinatorial description of the Schreier graphs of the mother groups.
In \cref{sec:NW} we give a generalization of the Nash-Williams resistance bound for collections of non-disjoint cutsets.
Unlike the Nash-Williams bound, the generalized version always achieves the actual resistance if the correct cutsets and weights are used.
While this generalization is not difficult, we have not found a reference for it, and it is of some independent interest.
Finally, in \cref{sec:cutsets} we define a collection of cutsets in $\cG_{d,\mbar}$, assign them weights and deduce \cref{T:resistance}.

\section{Combinatorial description of the graphs}
\label{sec:graphs}

In this section we give a more explicit description of the Schreier graphs $\cG_{d,\mbar}$ and $\cG_{d,\mbar,n}$.
Recall that a vertex $x$ of $\cG_{d,\mbar}$ is an end in $\cE_0$ of $T_\mbar$, and so is naturally described by a sequence $(x_i)_{i\geq0}$ where $x_i\in[m_i]$ such that eventually $x_i=0$.
For $\cG_{d,\mbar,n}$ the vertices are finite sequences $x_{n-1}\dots x_1 x_0$.

We write $x\sim y$ to denote that $x$,$y$ are connected by an edge.
Edges are of $d+2$ different \textbf{types}, denoted by $t\in\{-1,0,\dots,d\}$, corresponding to the types of the generator associated with the edge.
In all cases, an edge connects vertices $x$ and $y$ which differ only in a single coordinate (though not all such pairs are connected).
We denote that coordinate by $k=k(x,y)$, so that $x_k\neq y_k$, and $x_i=y_i$ for all $i\neq k$.
For such a pair $x,y$, we have that
\begin{itemize}[nosep]
\item $(x,y)$ is an edge of type $-1$ if $k=0$.
\item $(x,y)$ is an edge of type $t\ge0$ if $k>0$, and $x_{k-1}=y_{k-1}\neq 0$,
  and moreover there are precisely $t$ indices $i<k-1$ for which $x_i\neq 0$.
\item Otherwise, $(x,y)$ is not an edge.
\end{itemize}
For example, $x=0340020$ is connected to $y=0140020$ by an edge of type $1$
(here $k=5$), since $x_1$ and $x_4$ are non-zero.
The same $x$ is not adjacent to $z=0320020$, since $x_{k-1} = 0$.
See \cref{fig:smallgraphs} for some small examples.

\begin{figure}
  \centering
  \includegraphics[width=.8\textwidth]{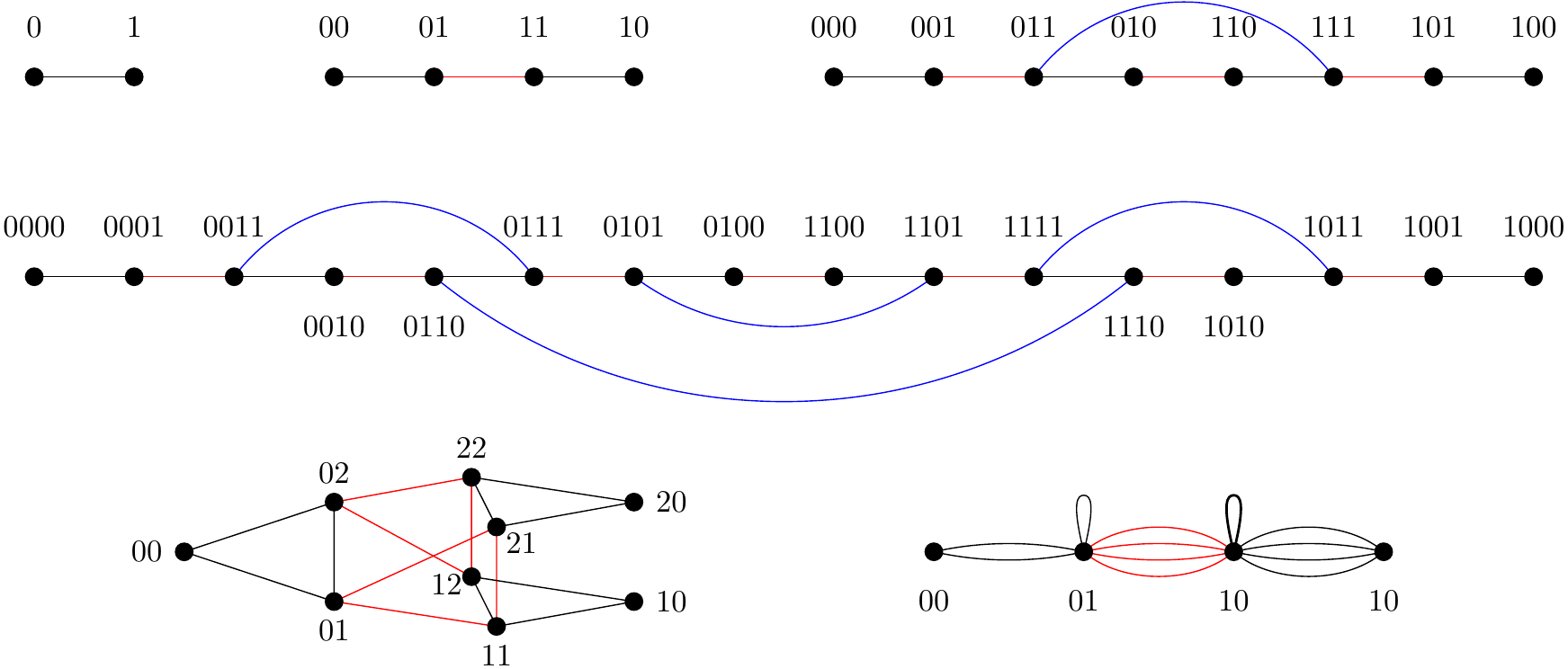}
  \caption{Top and middle rows: The graphs $\cG_{d,2,n}$ for $d=1$ and $n=1,2,3,4$.
    Bottom row: The graph $\cG_{1,3,2}$ and its quotient $\cG_{1,2,2}$ with multiple edges and self loops.
    Edge colour denotes its type: Black for $t=-1$, red for $t=0$ and blue for $t=1$.
    The bold loop on the bottom right has multiplicity 4, not all of the same type, but of course has no effect on resistances.
    Vertices are laid out according to the linear order $\hat x$ from left to right.}
  \label{fig:smallgraphs}
\end{figure}

Clearly the graphs $\cG_{d,\mbar,n}$ are monotone in $d,m_i,n$: reducing
any $m_i$ restricts to a subset of the vertices, while reducing $d$ to $d'$ removes all edges of type $t>d'$.
Extending a vertex of $\cG_{d,\mbar,n}$ by $0$s
gives a vertex of $\cG_{d,\mbar,n'}$ for any $n'>n$.  Extending by
infinitely many
$0$s gives a vertex of  $\cG_{d,\mbar}$.  This gives a canonical embedding of
$\cG_{d,\mbar,n}$ in the graphs for larger $n$ and in $\cG_{d,\mbar}$.


Clearly for each $i$, the graphs $\cG_{d,\mbar}$ and $\cG_{d,\mbar,n}$ are also invariant to permuting the letters $\{1,\dots,m_i-1\}$.
Consider two vertices $x,y$ to be \textbf{equivalent} if they have the same non-zero coordinates, i.e.\ $\{ i : x_i \neq 0\} = \{i :y_i \neq0\}$.
From each such equivalence class we take as representative the vertex in $\{0,1\}^n$.
The \textbf{hamming weight} of a vertex $x$, denoted $|x|$ is the number of non-zero coordinates.
The equivalence class of $x$ has $\prod_{i : x_i\neq 0} (m_i-1)$ vertices.
The projection from $\cG_{d,\mbar,n}$ to $\{0,1\}^n$ is denoted by $\pi$.
In the limit $n\to\infty$, this projection extends to a projection from $\cG_{d,\mbar}$ to $\oplus_{i\in \N} \{0,1\}$, namely the set of $\{0,1\}$ sequences with finitely many ones.
Since these are the vertices of $\cG_{d,2,n}$ (or $\cG_{d,2}$), we can see $\pi$ as a map from $\cG_{d,\mbar,n}$ to $\cG_{d,2,n}$, which preserves much of the graph structure.

The resistances we shall consider are between sets that are themselves invariant to such permutations, and thus can be studied by looking at resistances on the quotient graph.
Many edges become self-loops under this projection, and do not affect the resistance.
If $(x,y)$ is an edge where $x_k=0$ and $y_k\neq 0$, then the edge $x,y$ is projected onto an edge of $\cG_{d,2,n}$, so $\pi$ maps $\cG_{d,\mbar,n}$ to $\cG_{d,2,n}$ with self-loops.
Each edge of $\cG_{d,2,n}$ can have multiple preimages under $\pi$.
The number of preimages of an edge $(x,y)$ is $\prod_{i: y_i\neq 0} (m_i-1)$ (assuming $x_k=0$ and $y_k\neq 0$).
We will therefore consider the graph $\cG_{d,2,n}$ where edges have conductance given by this multiplicity.
We remark that $\prod_{i: x_i\neq 0} (m_i-1)$ and $\prod_{i: y_i\neq 0} (m_i-1)$ differ by a bounded factor of $m_j-1$ for some $j$, so up to constant factors either can be used for the conductance of the edge.
The same holds for the projection from $\cG_{d,\mbar}$ to $\cG_{d,2}$.
We will work below primarily with the graphs $\cG_{d,2,n}$ and $\cG_{d,2}$ with edge conductances coming from these multiple preimages.

The graphs in the case $d=0$ are particularly simple.
Each vertex of $\cG_{0,2,n}$ is incident to one edge of type $-1$ and one edge of type $0$, except for the root $o = 00\dots0$ and one other vertex $o' = 10\dots0$ which have degree 1.
It is not hard to verify that the graph $\cG_{0,2,n}$ is a path of length $2^n-1$ from $o$ to $o'$.
Since increasing $d$ does not remove any edges, this path is contained in $\cG_{d,2,n}$ for any $d$.
It will be useful to
keep track of the position of a vertex $x$ along this path, which shall be
denoted $\hat x\in[0,2^n-1)$.


Given a finite binary string $x=x_nx_{n-1}\dots x_0$ we define its \textbf{linear position} $\hat{x}_n \hat{x}_{n-1}\dots \hat{x}_0$ by $\hat{x}_k = \sum_{i=k}^{n} x_k \pmod{2}$.
Note that this map is a bijection from $\{0,1\}^n$ to itself, and the inverse transform is given by $x_k=\hat{x}_k + \hat{x}_{k+1} \pmod{2}$.
This extends to infinite sequences $x\in \bigoplus_{i\in\N} \{0,1\}$, since
the infinite sum contains finitely many ones.
We also define the projection $\hat\pi$ on $\cG_{d,\mbar}$ and $\cG_{d,\mbar,n}$ by composing this transformation with the projection $\pi$, namely $\hat\pi(x) = \widehat{\pi(x)}$.
For example if $x=0340020$, then $\pi(x)=0110010$ and $\hat\pi(x) = 0100011$.

We remark that $\bigoplus_{i\in \N} \{0,1\}$ is naturally in bijection with $\N$, where any integer corresponds to its binary expansion.
This gives a natural order on $\cG_{d,2}$ and a partial order on $\cG_{d,\mbar}$, where $x<y$ if $\hat\pi(x)<\hat\pi(y)$ as integers.
The root is the unique minimal vertex, mapped to $\hat{o} = 0$.

Recall that for a vertex $x$ we denote by $\ell_i(x)$ the position of the $(i+1)$-th non-zero digit in $i$ from the right, and we let $\ell_{-1}(x)=-1$.
Thus for an edge $e=(x,y)$ of type $t$ we have $\ell_i(x)=\ell_i(y)$ for all $i\leq t$.
We denote the part of $x$ (or $y$) strictly to the left of position $k$ by $z(x,y)$.
When $x$ and $y$ are fixed, we shorten notation and denote the above simply by $k$, $\ell_i$ and $z$.
We shall make use of the following description of edges in terms of the linear order on vertices:



\begin{remark}\label{C:0-1edges}
  Edges of type $-1$ and $0$ always connect adjacent points in the linear order.
  That is, $0$ is connected to $1$ by an edge of type $-1$, and every other $x\in \{0,1\}^*$ $\hat{x}$ is connected to $\hat{x}\pm 1$ by edges of types $-1$ and $0$.
  The edge types alternate along the resulting path.
  See \cref{fig:smallgraphs}.
\end{remark}

\section{Weighted Nash-Williams}
\label{sec:NW}

In this section we give a generalization of the classical Nash-Williams bound on resistances in electrical networks, which applies for collections of not-necessarily disjoint cutsets.
This generalization, which is also of some independent interest,  will be used to give lower bounds on resistances in the Schreier graphs of the mother groups.


We assume here the reader has basic familiarity with the theory of electrical networks, and refer the reader to e.g. \cite{LP,DS} for detailed background.
We recall the notations we use below.
An electrical network is a graph $G=(V,E)$ with edge weights or conductances $C_e\in\R_+$.
The resistance of an edge is denoted $R_e = C_e^{-1}$.
An unweighted graph is seen as a network with $C_e\equiv 1$.
We denote the resulting effective resistance between vertices $a,b$ by $\Res(a,b)$.
This is extended to resistance between sets $A,B$, denoted $\Res(A,B)$.

Recall the classical Nash-Williams inequality:
In any graph $G$ with vertices $a,b$, if $\{S_i\}_{i\in I}$ are disjoint edge cutsets (i.e.\ $S_i$ separates $a$ from $b$), then $\Res(a,b) \geq \sum |S_i|^{-1}$.
This extends in the natural way to the resistance between sets $A,B$, as well as to electrical networks, where $|S_i|$ is replaces by the total conductance of $S_i$.
In general networks, there is no collection of disjoint cutsets for which this bound achieves the actual resistance.
There is not even a bound on how far from $\Res(A,B)$ the
optimal collection of cutsets might be.  However, it turns out that there is
a weighted version of Nash-Williams that can achieve the resistance on any
graph, which we describe below.

Let $\{S_i\}_{i\in I}$ be some collection of cutsets (not necessarily disjoint).
A \textbf{resistance allocation} is an assignment, where each edge splits its resistance between the cutsets containing it.
More explicitly, for each edge $e$ and $i\in I$, we have partial resistances $R_{e,i}\geq 0$ which satisfy $\sum_i R_{e,i} \leq R_e$ so that $R_{e,i}=0$ if $e\not\in S_i$.
Define the \textbf{split conductance} of $S_i$ by $C(S_i) = \sum_{S_i} R_{e,i}^{-1}$.

The following is the generalization of Nash-Williams to non-disjoint cutsets.
While it is fairly simple to prove, we are not aware of a reference for it in the literature.

\begin{prop}\label{P:WNW}
  With the above notations, for any collection of cutsets and resistance allocation we have $\Res(A,B) \geq \sum_i C(S_i)^{-1}$.
  Moreover, in any finite graph, $\Res(A,B)$ is the supremum of $\sum_i C(S_i)^{-1}$ over all resistance allocations on some collection of cutsets.
\end{prop}

A particular way of allocating resistances is to assign each cutset $S_i$ a weight $K_i\ge 0$ and allocate resistances in proportion to these weights.
Formally this means to set $R_{e,i} = \frac{R_e K_i}{\sum_{S_j\ni e} K_j}$.
Plugging this in yields the following bound:

\begin{corollary}
  For any collection $\{S_i\}_{i\in I}$ of cutsets between $A,B$, and any non-negative weights $(K_i)$, we have
  \[
    \Res(A,B) \geq \sum_i \left(
      \sum_{e\in S_i} \frac{\sum_{S_j\ni e} K_j}{R_e K_i} \right)^{-1}.
  \]
\end{corollary}

\begin{remark}
  If the graph is infinite, the weighted cutset method still gives a lower bound on the resistances.
  If we consider the resistance from a vertex (or set) to infinity, $\Res(A,\infty)$ is a limit of the resistance to the complement of an arbitrary exhaustion $G_n$.
  It follows that $\Res(A,\infty)$ is again the supremum over weighted cutsets as above.
  In particular, a graph is recurrent if and only if there exist a collection of cutsets $S_i$ and resistance allocations $R_{e,i}$ such that $\sum C(S_i)^{-1} = \infty$.
  We omit further details.
\end{remark}

\begin{proof}[Proof of \cref{P:WNW}]
  Assume first that there are only finitely many cutsets $S_i$ in the collection.
  Given a resistance allocation, we construct a new network $G'$ from $G$, where each edge $e$ is replaced by several edges in series, with resistances
  $\{R_{e,i}\}_{i\in I}$.
  (The order of these edges in the series is arbitrary.)
  Since $\sum_i R_{e,i} \leq R_e$, effective resistances in $G'$ are all smaller than in $G$.
  In $G'$ we can construct a collection of disjoint cutset:
  For each $i$ take the edges of resistance $R_{e,i}$.
  The classical Nash-Williams applied to these disjoint cutsets gives the claimed bound.
  If there are infinitely many cutsets just note that any finite partial sum gives a finite resistance allocation, and thus gives a lower bound on $\Res(A,B)$.

  To see that some weighted cutsets achieve the resistance, we give an explicit construction.
  Suppose first that $G$ is finite.
  Consider the induced equilibrium voltage with $V=0$ on $A$ and $V=1$ on $B$, and let $f$ be the equilibrium flow from $a$ to $b$, so that $f(x,y) = \frac{V_y-V_x}{R_{xy}}$.
  Let $0 = a_0 < a_1 < \dots < a_m = 1$ be the different values taken by $V$.
  For each $i\le m$, let $U_i = \{x\in G : V_x < a_i\}$, so that $A\subset U_i$ and $B \subset U_i^c$.
  Define the cutsets $S_i$ of edges $x,y$ with $x\in U_i$ and $y\not\in U_i$.

  We use the weighted resistance allocation as defined above.
  Assign $S_i$ weight $K_i = a_i-a_{i-1}$, so that $\sum K_i = V_B-V_A=1$.
  For an edge $e=(x,y)$ with $V_x<V_y$ we have that
  \[ \sum_{j : e\in S_j} K_j = V_y-V_x, \]
  and therefore $R_{e,i}^{-1} = R_e^{-1} \frac{V_y-V_x}{K_i} = f(e)/K_i$.
  Since the flow is always in the direction of increasing voltage, the total flow across any cutset $S_i$ is exactly $1/\Res(A,B)$.
  Thus $C(S_i)^{-1} = K_i \Res(A,B)$.
  Summing over $i$ we get the claim.
\end{proof}

\section{Cutsets in $\cG$}
\label{sec:cutsets}


We now use the linear order on vertices of $\cG = \cG_{d,2,n}$ or $\cG_{d,2}$ to define a collection of cutsets.
We remind that we work here with the graphs resulting from projecting $\cG_{d,\mbar,n}$ so that edges have unequal conductances.
The conductance of an edge $(x,y)$ is either $\prod_{x_i=1} (m_i-1)$ or the corresponding product for $y$.
The two are equivalent up to a bounded multiplicative factor.

For $\hat a\in \N$ we let $S_a$ be the set of all edges $(x,y)$ with $\hat x < \hat a\leq \hat y$.
Note that these cutsets are not disjoint for $d>0$.
(If $d=0$, then $\cG$ is a path, and each of these cutsets is a single edge.)
As with $x$, each $\hat a$ is associated with a sequence $a\in\{0,1\}^*$.
Note that even for general sequences $m$ we take $a\in\{0,1\}^*$.

For our analysis, it will be convenient to enlarge these cutsets slightly.
Edges of type $-1$ and $0$ will not be added to the cutsets.
However, to some cutsets we will add an edge of type $1$ and possibly several edges of type $2$, as described below.
The enlarged cutsets will be denoted $\bar{S}_a$, and are defined formally after the proof of \cref{L:e_in_Sa}.

To a cutset $\bar S_a$ we associate weight $\beta^a := \prod_{i: \, a_i=1}\beta_i$, where $\beta_i=1/(m_i-1)$.
It will be convenient to extend this notation to general sequences by $\beta^x := \prod_i \beta_i^{x_i}$.
We also use below the notation $\beta^{-a} = (m-1)^a = 1/\beta^a$.
We then allocate the resistance $R_e$ of an edge $e$ between the cutsets in proportion to their weight,
i.e.\ for $e\in \bar S_a$ let
\[
  R_{e,a} = \frac{R_e \beta^a}{\sum_{a : e\in \bar S_a} \beta^a}.
\]

Our immediate goal is therefore to understand which of the cutsets $S_a$ include a given edge and which edges are included in any cutset.
The enlarged cutsets $\bar{S}_a$ will be defined so that $R_{e,a}$ is easier to analyse.

\begin{lemma}\label{L:e_in_Sa}
  Consider an edge $e=(x,y)$ of type $t$ with $\hat x < \hat y$.
  \begin{enumerate}[nosep]
  \item If $t\in\{-1,0\}$, then $e\in S_a$ if and only if $\{\hat x,\hat y\} = \{\hat{a}-1, \hat a\}$ as an unordered pair.
  \item If $t=1$ and $a \neq x$, then we have $e\in S_a$ if and only if
    $a_i=x_i=y_i$ for all $i>\ell_0$ except $i=k$.
  \item If $t=2$ and $e\in S_a$, then $a_i=x_i=y_i$ for all $i>\ell_1$ except $i=k$.
  \end{enumerate}
\end{lemma}

\begin{figure}
  \centering
  \includegraphics[width=0.7\textwidth]{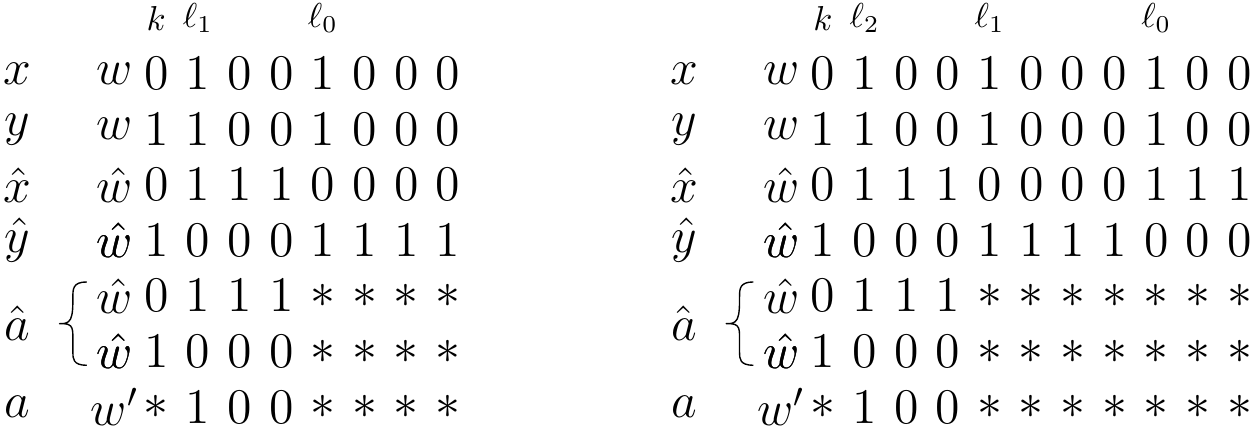}
  \caption{Examples for \cref{L:e_in_Sa}.
    Left: $(x,y)$ is a type 1 edge.
    Both begin with a common prefix $w$, and differ only in position $k$, with a one immediately afterwards.
    If $w$ has an even number of ones then $\hat x < \hat y$, otherwise it would be reversed.
    If $(x,y)\in S_a$ then $\hat x < \hat a \leq \hat y$, and so $\hat a$ must take one of two forms, depending on its $k$th digit.
    The $*$'s indicate digits that could take any value in $\{0,1\}$.
    However, in the first form of $\hat a$, if all $*$'s are $0$ then $\hat x=\hat a$, which is excluded.
    In either case, $a$ agrees with either $x$ or $y$ up to position $\ell_0$.
    Right: a type 2 edge.
    Here, $\hat x < \hat a \leq \hat y$ implies that $\hat a$ and $a$ have the form shown.
    However, even more cases are excluded, for example if $\hat a = \hat{w}10001111010$ then $\hat a>\hat y$ and $(x,y)\not\in S_a$.
    The enlarged cutsets $\bar{S}_a$ contain $(x,y)$ whenever $a$ has the form above, even if $\hat a \not\in(\hat x,\hat y]$.}
  \label{F:cutset}
\end{figure}

See \cref{F:cutset} for examples of the type 1 and type 2 cases.
Note that in the case $t=1$ we have $k=\ell_1+1 > \ell_0$.
In that case the condition on the digits of $a$,$x$,$y$ is satisfied when $a=x$ but $e\not\in S_a$ due to the strict inequality in the definition of $S_a$.
In the case $t=2$ we do not provide a sufficient criterion for $e\in S_a$ but only a necessary condition.
In the case of $t=2$ one could give a necessary and sufficient condition for $e$ to be in $S_a$, which would be more cumbersome and would not lead to a significant improvement in the estimates below.

\begin{proof}
  The cases $t=-1$ and $t=0$ follow immediately from \cref{C:0-1edges}.

  Let $e=(x,y)$ be a type-1 edge with $\hat{x}<\hat{y}$.
  Since $k=1+\ell_1$ is the unique index where $x_k\neq y_k$, we have that $x$ and $y$ have the following form (from left to right):
  They start with the same sequence of bits $w$, until position $k$.
  At position $k$ one of them is $0$ and the other is $1$.
  Which is $0$ depends on the parity of the number of $1$s in $w$.
  At position $\ell_1 = k-1$ both are $1$, and the rest of the bits are $0$ except for a single position $l_0$ where
  also $x_{\ell_0}=y_{\ell_0}=1$.
  See \cref{F:cutset}

  Therefore their linear order representations $\hat{x},\hat{y}$ have the following structure:
  Both begin (on the left) with $\hat{w}$ till position $k$.
  At position $k$, $\hat{x}=0$ and $\hat{y}=1$ (since we assumed $\hat{x}<\hat{y}$);
  This is followed in $\hat{x}$ by $\ell_1-\ell_0$ ones, and $\ell_0$ zeros.
  In $\hat{y}$, the final ones and zeros are reversed: there are $\ell_1-\ell_0$ zeros followed by $\ell_0$ ones.

  Since we assume $a \neq x$, then also $\hat a \neq \hat x$.
  Therefore the assumption $(x,y)\in S_a$ is equivalent to $\hat{x} \leq \hat{a} \leq \hat{y}$.
  Then $\hat{a}$ must agree with both $\hat{x}$ and $\hat{y}$ in all positions left of $k$.
  If $\hat{a}_k=0$ then $\hat{a} \leq \hat{y}$ must hold, and the condition $\hat{x} \leq \hat{a}$ is equivalent to the next $\ell_1-\ell_0$ digits being $1$ (and the final $\ell_0$ digits can be anything).
  Similarly, if $\hat{a}_k=1$ then $\hat{x} \leq \hat{a}$ must hold, and the condition $\hat{a} \leq \hat{y}$ is equivalent to the next $\ell_1-\ell_0$ digits all being $0$.

  Converting this description of $\hat{a}$ to $a$ yields the claim for the case $t=1$ (recall $a_i=\hat{a}_i+\hat{a_{i+1}}$ mod $2$).

  \medskip

  The case $t=2$ is similar.
  Let $e=(x,y)$ be a type-2 edge with $\hat{x}<\hat{y}$.
  Since $k=1+\ell_2$ is the unique index where $x_k\neq y_k$, we have that $x$ and $y$ have the following form (from left to right):
  They start with the same sequence $w$, until position $k$.
  At position $k$ one of them is $0$ and the other is $1$.
  Subsequently, their non-zero digits are precisely in positions $\ell_2,\ell_1,\ell_0$.

  In the linear order representation, $\hat{x},\hat{y}$ have the
  following structure:
  Both begin (on the left) with $\hat{w}$ till position $k$.
  At position $k$, $\hat{x}=0$ and $\hat{y}=1$.
  This is followed in $\hat{x}$ by a block of $1$s, a block of $0$s, and another block of $1$s, and in $\hat{y}$ by blocks of the same lengths, but starting and ending with $0$s.

  Suppose $(x,y)\in S_a$, and in particular $\hat{x} \leq \hat{a} \leq \hat{y}$.
  Then $\hat{a}$ must agree with both $\hat{x}$ and $\hat{y}$ in all positions left of $k$.
  If $\hat{a}_k=0$ then $\hat{a} \leq \hat{y}$ holds, and the assumption $\hat{x} \leq \hat{a}$ implies that the next
  $\ell_2-\ell_1$ digits of $\hat{a}$ are all $1$s.
  Similarly, if $\hat{a}_k=1$ then the assumption $\hat{a} \leq \hat{y}$ implies that the next $\ell_2-\ell_1$ digits are $0$s.
  Converting this description of $\hat{a}$ to $a$ yields the claim for the case $t=2$.
\end{proof}

In light of \cref{L:e_in_Sa} we define the \textbf{enlarged cutsets} $\bar S_a$ which contain all edges $(x,y)$ which satisfy the condition in the corresponding clause of the \cref{L:e_in_Sa}.
Explicitly, an edge $(x,y)$ of type $t$ is in $\bar S_a$ if
\begin{itemize}[nosep]
\item $t\in\{-1,0\}$, and $\hat x=\hat{a}-1$, $\hat y=\hat{a}$.
\item $t=1$ and $a_i=x_i=y_i$ for all $i>\ell_0$ except possibly $i=k$, or
\item $t=2$ and $a_i=x_i=y_i$ for all $i>\ell_1$ except possibly $i=k$.
\end{itemize}
From here on we work with the cutsets $\bar S_a$.

\begin{lemma}\label{L:cutset_count}
  For an edge $e=(x,y)$ of type $t$, we have that
  \[
  \sum_{a : e\in \bar S_a} \beta^a \asymp
  \begin{cases}
    \beta^x & t=-1 \text{ or } t=0, \\
    \beta^x \prod_{i\leq\ell_0} (1+\beta_i) & t=1, \\
    \beta^x \prod_{i\leq\ell_1} (1+\beta_i) & t=2,
  \end{cases}
  \]
  where the implicit constants depend only on $\max \bar m$.
\end{lemma}

\begin{proof}
  The case $t\leq0$ is trivial since $a=x$ or $a=y$ and $\beta^x$ and
  $\beta^y$ differ by the bounded ratio $\beta_k$.

  In the case $t=1$, from the definition of $\bar S_a$ we have
  $a_i=x_i=y_i$ for $i>\ell_0$, except $i=k=\ell_1+1$.  In
  particular, $a_i=0$ for
  $i\in(\ell_0,\ell_1)$.  We are interested in the sum of $\beta^a$ over
  all $a$ with $e\in \bar S_a$.
  Since for $i\le \ell_0$ we can have any combination of $0$s and $1$s, this satisfies
  \[
  \sum_{a : e\in \bar S_a} \beta^a = \Big( \prod_{i>k} \beta_i^{x_i} \Big)
  (1+\beta_k) \beta_{\ell_1} \Big( \prod_{i\leq \ell_0} (1+\beta_i) \Big).
  \]
  Since $\beta$ is bounded from $0$ and above, and since $x_i=0$ for $i<k$  except $i=\ell_0,\ell_1$, the first term in this product is (up to constants) $\beta^x$, and the next two terms are bounded, giving the lemma.

  The case $t=2$ is almost identical, with $\ell_1$ replacing $\ell_0$.
\end{proof}

The next step is to compute the total conductance of a cutset.
Whereas previously we were interested in which cutsets contain an edge, now this requires the dual question: which edges are in a cutset.
Each cutset contains a unique edge of type 0 or -1, but can contain more edges of higher types.
For any $a$, the conductance of $\bar S_a$ is given by $C_a := \sum_{e\in \bar S_a} R_{e,a}^{-1}$.
Note that for an integer $a$ we have $[\log_2 a] = \max \{i : a_i\neq 0\}$.

\begin{lemma}\label{L:C_a_1}
  In the degree 1 mother group, we have that
  \[ C_a \asymp \beta^{-a} \prod_{i<\log_2 a} (1+\beta_i), \]
  where the constants depend only on $\max \bar m$.
\end{lemma}

\begin{proof}
  Fix some $a\in\N$, and consider the cutset $\bar S_a$.
  The cutset contains exactly one edge of type $-1$ or $0$ (with $\hat{x}+1=\hat{y}=\hat{a}$).
  This edge has conductance $R_e^{-1} = \beta^{-x}$ or $\beta^{-y}$ (equivalent up to constants to $\beta^{-a}$), and assigns all of it to the cutset $\bar{S}_a$.
  The contribution to $C_a$ from edges $e=(x,y)$ of type 1 is given (using \cref{L:cutset_count}) by
  \begin{align*}
    \sum_{e\in \bar S_a} R_{e,a}^{-1}
    &= \sum_{e\in \bar S_a}
    \frac{\sum_{b : e\in \bar S_b} \beta^b} {R_e \beta^a} \\
    &\asymp \sum_{e\in \bar S_a}
    \frac{\beta^x \prod_{i\leq\ell_0} (1+\beta_i)} {\beta^x \beta^a} \\
    &= \beta^{-a} \sum_{e\in \bar S_a} \prod_{i\leq\ell_0} (1+\beta_i).
  \end{align*}

  We now argue that, from the definition of $\bar S_a$, for each $\ell_0\leq \log a$ there is a unique edge $(x,y)$ of type $1$ in $\bar S_a$ with the given $\ell_0$.
  For $\ell_0>\log a$ there are no edges of type $1$ in $\bar S_a$.
  To see this, let us fix $\ell_0$ and $a$, we try to recover the edge $(x,y)$.
  First we find $\ell_1$, which must be the minimal  $i>\ell_0$ with $a_i=1$.
  This is the only choice, since $a_{\ell_1}$ cannot be $0$ if $(x,y)\in\bar S_a$, and since $a_i=0$ for all
  $i\in(\ell_0,\ell_1)$.
  Such $\ell_1$ can be found if and only if $\ell_0 < \log a$.
  We now can identify $x$ and $y$, since $x_i=y_i=0$ for all $i\leq \ell_1$ except $i=\ell_0,\ell_1$ where $x_i=y_i=1$.
  Moreover, $x_i=y_i=a_i$ for all $i>\ell_1+1$.
  Finally, for $i=\ell_1+1$ we have that $x_i$,$y_i$ are $0$ and $1$ in some order.

  Thus we have
  \[
    C_a \asymp \beta^{-a} + \beta^{-a} \sum_{\ell_0\leq \log_2 a} \prod_{i\leq \ell_0} (1+\beta_i)
  \]
  where the first term is from the edge of type $0$ or $-1$ and the sum from edges of type $1$.
  Since $\beta_i$ are bounded away from $0$, the sum is dominated up to a constant factor by the largest term and we get
  \[
    C_a \asymp \beta^{-a} \prod_{i < \log_2 a} (1+\beta_i)
  \]
  as claimed.
\end{proof}

\begin{lemma}\label{L:C_a_2}
  In $\cG_{2,\mbar}$ we have that $C_a \asymp \beta^{-a} (\log_2 a) \prod_{i<\log_2 a} (1+\beta_i)$, where the constants depend only on $\max \bar m$.
\end{lemma}

\begin{proof}
  This is very similar to the proof of \cref{L:C_a_1}.
  The main change and additional contribution now is from edges of type $2$, and so we need to understand edges of type $2$ in $\bar S_a$.
  We claim that for any $\ell_1 < \log_2 a$, there are exactly $\ell_1$ edges $e\in\bar S_a$ with that $\ell_1$.
  To see this note that given $\ell_1 <
  \log a$, and any $\ell_0<\ell_1$, there is a unique edge $x,y$ of type $2$ with those $\ell_0,\ell_1$ in the cutset.
  (This is just as the type $1$ case in the previous lemma.)
  Thus the contribution to $C_a$ from edges of type 2 is up to constants
  \[
    \beta^{-a} \sum_{\ell_1 < \log_2 a} k \prod_{i\leq \ell_1} (1+\beta_i)
    \asymp \beta^{-a} (\log_2 a) \prod_{i < \log_2 a} (1+\beta_i),
  \]
  since the sum is again dominated by its largest term.
  This dominates the contribution from edges of type $-1,0,1$, and so gives the claimed total conductance.
\end{proof}

\begin{proof}[Proof of \cref{T:resistance}]
  To bound the resistance from $\pi^{-1}([0,2^s)$ to   $\pi^{-1}([2^t,\infty))$, we separate cutsets into groups according to   $[\log_2 a] \in [s,t)$.
  For $k\in[s,t)$, there are $2^k$ choices for $a$ with $[\log_2 a]=k$.

  In the case $d=1$,
  the contribution from the cutsets $\bar S_a$ with $[\log_2 a] = k$ is given by \cref{L:C_a_1}:
  \[
    \sum_{a=2^k}^{2^{k+1}-1} C_a^{-1}
    \asymp
    \sum_{a=2^k}^{2^{k+1}-1} \beta^a \prod_{i\leq k}
    (1+\beta_i)^{-1} = \beta_k.
  \]
  Since $\beta_k$ is bounded, the total resistance is at least $c(t-s)$.

  In the case $d=2$, \cref{L:C_a_2} gives an extra factor of $1/k$ in the $k$th term, so
  \[
    \sum_{a=2^k}^{2^{k+1}-1} C_a^{-1} \asymp
    \sum_{a=2^k}^{2^{k+1}-1} \beta^a k^{-1} \prod_{i\leq k}
    (1+\beta_i)^{-1} = \beta_k k^{-1}.
  \]
  Therefore
  \[
    \Res\big(\hat\pi^{-1}[0,2^s), \hat\pi^{-1}[2^t,\infty)\big) \geq
    \sum_{k=s}^{t-1} c/k
    \asymp \log(t)-\log(s)
  \]
  as Claimed.
\end{proof}

\begin{proof}[Proof of \cref{T:recurrent}]
  Fix $s=0$ in \cref{T:resistance}.
  We find that $\Res\big(\hat\pi^{-1}(0), \hat\pi^{-1}[2^t,\infty)\big)$ is unbounded as $t\to\infty$.
  Recurrence of the graph follows.
\end{proof}

\paragraph{Acknowledgments.}
G.A. was supported by Israeli Science Foundation grant 957/20.
OA is supported in part by NSERC.
BV was supported by the Canada Research Chair program and the NSERC Discovery Accelerator grant.

\bibliographystyle{plain}
\bibliography{automata}

\end{document}